\documentclass[11pt]{amsart}

\usepackage[top=2.5 cm,bottom=2 cm,left=1.5 cm,right=1.5 cm]{geometry}

\usepackage{amssymb,amsfonts,graphicx,amsmath,amsthm,amstext,latexsym,amsfonts}
\usepackage{graphicx,epsfig}
\usepackage{bbm}
\usepackage{t1enc}
\usepackage{mathrsfs}
\usepackage{subfig}
\usepackage{hyperref}
\usepackage{a4wide}
\usepackage{mathrsfs}

\usepackage{tikz}
\usetikzlibrary{intersections}
\theoremstyle{plain}

\numberwithin{equation}{section}
\newtheorem{theorem}{Theorem}[section]
\newtheorem{proposition}[theorem]{Proposition}
\newtheorem{lemma}[theorem]{Lemma}
\newtheorem{definition}[theorem]{Definition}
\newtheorem{example}[theorem]{Example}
\newtheorem{remark}[theorem]{Remark}

\usepackage{color}
\definecolor{darkred}{rgb}{0.8,0,0}
\definecolor{darkblue}{rgb}{0,0,0.7}
\definecolor{darkgreen}{rgb}{0,0.4,0}

\newcommand{\eps}{\varepsilon}

\newcommand{\R}{{\mathbb R}}

\newcommand{\W}{{\mathcal W}}

\newcommand{\V}{{\mathcal V}}

\newcommand{\un}{{\rm 1\kern -2.5pt l}}
\newcommand{\tr}{{\rm Tr}}

\def\u{\mathbf{u}}
\def\uu{\mathbf{uu}}
\def\vv{\mathbf{v}}
\def\yy{\mathbf{y}}
\def\n{\mathbf{n}}

\def\xx{\mathbf{x}}
\def\aa{\mathbf{a}}
\def\bb{\mathbf{b}}

\def\uu{\mathbf{u}}
\def\vv{\mathbf{v}}

\def\eps{\varepsilon}
\def\R{{\mathbb R}}

\def\M{{\mathcal M}}

\def\H{{\mathcal H}}

\def\eps{\varepsilon}
\def\R{{\mathbb R}}

\def\e{{\mathcal F}}

\def\M{{\mathcal M}}

\def\H{{\mathcal H}}

\def\F{{\mathcal F}}

\def\E{{\mathbb{E}}}

\def\argmin{\mathop{{\rm argmin}}\nolimits}

\def\Tr{\mathop{{\rm Tr}}\nolimits}

\def\dv{\mathop{{\rm div}}\nolimits}

\def\sym{\mathop{{\rm sym}}\nolimits}

\def\u{\mathbf{u}}
\def\v{\mathbf{v}}
\def\e{\mathbf{E}}
\def\z{\mathbf{z}}
\def\v{{\bf v}}
\def\w{{\bf w}}
\def\e{{\bf e}}
\def\x{{\bf x}}

\def\Id{\mathbf{I}}

\def\wconv{\rightharpoonup}

\renewcommand{\epsilon}{\varepsilon}
\newcommand{\beeq}{\begin{equation}}
\newcommand{\eneq}{\end{equation}}
\newcommand{\bear}{\begin{array}}
\newcommand{\enar}{\end{array}}
\newcommand{\bema}{\begin{displaymath}}
\newcommand{\enma}{\end{displaymath}}

\newcommand{\beea}{\begin{eqnarray}}
\newcommand{\enea}{\end{eqnarray}}

\newcommand{\om}{\Omega}

\newcommand{\lab}[1]{ \label{#1} }

\def\Id{\mathbf{I}}

\def\wconv{\rightharpoonup}

\title[]{A new variational approach to 
linearization of\\ traction problems in Elasticity}
   \author[]{Francesco Maddalena, Danilo Percivale, Franco Tomarelli }

 \address{Politecnico di Bari, Dipartimento di Meccanica, Matematica, Management, via Re David 200, 70125 Bari, Italy}
 \email{francesco.maddalena@poliba.it}
 \address{Universit\`{a} degli Studi di Genova, Dipartimento di   Ingegneria Meccanica, Energetica, Gestionale e dei Trasporti (DIME),
  Piazzale Kennedy, Fiera del Mare, Padiglione D, 16129 Genova, Italy}
  \email{percivale@diptem.unige.it}
\address{Politecnico di Milano, Dipartimento di Matematica,  Piazza Leonardo da Vinci 32, 20133 Milano, Italy}
\email{franco.tomarelli@polimi.it}
\date{\today}  \subjclass{}
\begin{document}
 \maketitle
\begin{abstract}
A new energy functional for pure traction problems in elasticity has been deduced in \cite{traction} as the variational limit
of nonlinear elastic energy functional for a material body subject to an equilibrated force field: a sort of Gamma limit with respect to
the weak convergence of strains when a suitable small parameter tends to zero. This functional exhibits a gap that makes it different from the classical linear elasticity functional. Nevertheless a suitable compatibility condition on the force field  ensures coincidence of related minima and minimizers. Here we show some relevant properties of the new functional and prove stronger convergence of minimizing sequences for suitable choices of nonlinear elastic energies.
\end{abstract}
\begin{center}
\large\sffamily DRAFT
\end{center}
\tableofcontents
\begin{flushleft}
  {\bf AMS Classification Numbers (2010):\,} 49J45, 74K30, 74K35, 74R10.
\end{flushleft}
\vskip0.5cm
\section{Introduction}
This article is focussed on the properties of the functional
\begin{equation}
\label{DTfuncintro}
\F(\v)\ \,:=\ \,\displaystyle\min_{\mathbf W\in \M^{N\times N}_{skew}}\ \int_\Omega
\mathcal V_0\left( \,\xx,\,\mathbb E(\v) - \textstyle\frac{1}{2}\mathbf W^{2} \,\right)
\, d\xx\ \,-\ \,\mathcal L(\v) \ .
\end{equation}
In \eqref{DTfuncintro} and in the sequel we set: $ N=2,3,$
$\,\M^{N\times N}_{skew}$ denotes the set of skew-symmetric $N\!\times\! N$ real matrices,
$\Omega\subset\R^N$ is a Lipschitz open set representing the reference configuration of an hyperelastic material body undergoing pure traction, $\mathcal V_0(\xx,\cdot)$\,  are uniformly positive definite quadratic forms on square matrices, the vector field $\,\v$ in $H^1(\om,\R^N)$\, denotes a displacement and $\,\E(\v):= \textstyle\frac{1}{2}(\nabla\v^{T}+\nabla\v)$\, denotes the related linearized strain, while $\mathcal L(\v)$ represents the potential energy associated to displacement $\v$,
\begin{equation}
\label{globalequiintro1}
\mathcal L(\v)\ :=\
\int_{\partial\Omega} \mathbf f\cdot\mathbf v\,d\H^{N-1}+\int_{\Omega} \mathbf g\cdot\mathbf v\,d\xx\,, \qquad \
\mathbf f\in L^2(\partial \Omega;\R^N),\ \mathbf g\in L^{2}(\Omega)\,,
\end{equation}
here $\mathbf f$ and $\mathbf g$ are respectively the prescribed boundary and body force fields,
moreover we assume that the total load is equilibrated, say
\begin{equation}
\label{globalequiintro2}
\mathcal L(\z)\ =\
0
\qquad \forall \,\mathbf z \,:\  \mathbb E(\mathbf z)\equiv \mathbf 0\ .
\end{equation}
Motivations for studying functional $\mathcal F$ 
and its minimization over $\v$ in $H^1(\om,\R^N)$ rely on the variational asymptotic analysis developed in \cite{traction}, where we proved that
for pure traction problems in elasticity a gap arises between
the classical linearized elasticity functional $\mathcal  E$,
\begin{equation}
\label{linearelIntro}
{\mathcal  E}(\v):=\int_\Omega \mathcal V_0(\xx,\mathbb E(\v))\,d\x-\mathcal L(\v)\,,
\end{equation}
and
the rigorous variational limit
of nonlinear elastic energy of a material body subject to an equilibrated force field, since
this limit actually is functional $\mathcal F$, provided the load fulfils a suitable compatibility condition (see \eqref{compintro} and Theorem \eqref{mainth1} below).
\vskip0.5cm
The inequality $\mathcal  F (\v)\leq \mathcal  E(\v)$ for every $\v$ is straightforward.
nevertheless the two functionals cannot coincide: indeed
$\mathcal F (\v) =-\mathcal L(\v)< \mathcal E(\v)$ whenever $\v(\xx)=\frac 1 2 \mathbf W^2 \xx$
with
 $\mathbf W\neq \mathbf{0}$ skew symmetric matrix.\\
Notwithstanding this gap, in \cite{traction} we showed that
the two functionals $\mathcal  F $ and $\mathcal E$ have the same minimum and same set of minimizers
 when the loads are equilibrated and compatible (see Theorem \eqref{mainth1} below).
\\
In the case $N=2$ the gap between the two functionals can be better clarified as follows (see 
Remark 2.5 in \cite{traction} for more details ):
\begin{equation}\label{F=E-err}
 \mathcal F(\vv) \ = \ \mathcal E(\vv)
-\frac{1}{4}\left (\int_{\om}\mathcal V_0(\xx,\mathbf I)
d\xx\right )^{\!-1}\left [\left (\int_{\om}D\mathcal V_0(\xx,\mathbf I)\!\cdot\!\mathbb E(\v)
\,d\xx\right)^{\!-}\right ]^{2} ,
\end{equation}
where $\alpha^-=\max(-\alpha,0)$, thus
\vskip-0.4cm
\begin{equation*}
\F(\v)=\mathcal  E(\v) \qquad \hbox{if }N=2 \quad \hbox{and}\quad \int_{\om}D\mathcal V_{0}(x,\mathbf I)\cdot \mathbb E(\v)\,dx\ge 0\,.
\end{equation*}
Even more explicitly, if $N=2$, $\,\lambda,\, \mu> 0\,$ and
\begin{equation}\lab{GSV1}
\W(\xx,\mathbf F)=\left\{\begin{array}{ll}
 \mu |\mathbf F^{T}\mathbf F-\mathbf I |^{2}+\frac{\lambda}{2}|\ \hbox{\rm Tr}\ (\mathbf F^{T}\mathbf F-\mathbf I)|^{2}
 \quad & \hbox{if} \ \det \mathbf F>0,
 \vspace{0.2cm}\\
 +\infty\quad & \hbox{ otherwise,}\\
 \end{array}\right.
\end{equation}
then $\mathcal V_0(\xx,\mathbf B)\,=\,4\mu|\mathbf B|^{2}+2\lambda|\hbox{Tr}\mathbf B|^{2}$ and we get
\begin{equation}\label{astarGSV}
a_{*}^{2}(\v)\ =\ |{\om}|^{-1}\left (\int_{\om}\hbox{\rm div}\,\v\,d\xx\right)^{\!-}\ ;
\end{equation}
such evaluation in 2D approximately means that
for every displacement $\v$ such that the associated deformed configuration $
\yy (\Omega)$ is greater than the area of $\om$,
the global energy $\mathcal F(\v)$
provided by new functional $\mathcal F$
is the same as the one provided by classical linearized elasticity, say $ \mathcal E(\v)$.
\vskip0.3cm
The rigorous derivation of the variational theory of linear elasticity  (\cite{Gu}) from the theory of finite elasticity  (\cite{Lo},\cite{TN}) was achieved in \cite{DMPN} through arguments based on De Giorgi $\Gamma-$ convergence
theory, thus providing a mathematical justification of the classical elasticity in small deformations regime, at least for \textsl{Dirichlet or mixed boundary value problem}.
\\
In a more recent paper (\cite{traction}) we have focussed the analysis on the analogous
variational question related to Neumann type condition, say the \textsl{pure traction problem in elasticity}:  the case where the elastic body is subject to a system of equilibrated forces and no Dirichlet condition is assigned on the boundary.
\\
Referring to the open set $\Omega\subset\R^N,\ N=2,3,$ as the reference configuration of an hyperelastic material body, the stored energy due to a deformation  $\mathbf y$ can be  expressed as
a functional of the deformation gradient $\nabla\yy$ as follows
\begin{equation*}
\int_\Omega\mathcal \,\mathcal W(\xx, \nabla\yy)\,d\xx
\end{equation*}
where  $ \W:\om \times\! \M^{N \times N}\!\to \! [0, +\infty ]$ is a frame indifferent function, $\M^{N\times N}$ is the set of real $N\times N$ matrices and
$\mathcal W(\xx, \mathbf F)<+\infty$ if and only if $\det \mathbf F >0$.\\ Then due to frame indifference there exists a function $\mathcal V$ such that
\begin{equation*}
\W(\xx,\mathbf F)=\V(\xx,\textstyle{\frac{1}{2}}( \mathbf F^T \mathbf F - \mathbf I))\,,
\qquad
\ \forall\, \mathbf F\in \M^{N\times N},\  \hbox{ a.e. }\xx\in \om.
\end{equation*}
We set $\mathbf F=\Id +h\mathbf B$, where $h> 0$ is an adimensional small parameter and
\begin{equation*}
\mathcal V_{h}(\xx,\mathbf B):= h^{-2}\mathcal W(\xx,\Id+h\mathbf B).
\end{equation*}
We assume that the reference configuration has zero energy and is stress free, i.e. $$\W(\xx,\mathbf I)=0,
\quad D \W(\xx,\mathbf I)=\mathbf 0 \quad \hbox{for a.e. }\xx \in \om \,,$$ and that $\W$ is regular enough in the second variable,
then Taylor's formula entails
\begin{equation*}
\mathcal V_{h}(\xx,\mathbf B)=\mathcal V_{0}(\xx, \sym \mathbf B) +o(1) \qquad \hbox{ as } h \to 0_+
\end{equation*}
where $\sym \mathbf B:=\frac{1}{2}(\mathbf B^{T}+\mathbf B)$ and
\begin{equation*}
\mathcal V_{0}(\xx, \sym \mathbf B):=\frac{1}{2} \sym \mathbf B\, D^{2}\mathcal V(\xx,\mathbf 0) \,\sym \mathbf B.
\end{equation*}
If the deformation $\yy$ is close to the identity up to a small displacement, say $\mathbf y(\xx)= \xx+h\mathbf \v(\xx)$ with
bounded $\nabla \v\,$ then,
by setting $\,\E(\v):= \textstyle\frac{1}{2}(\nabla\v^{T}+\nabla\v)$\,, one easily obtains

\begin{equation}
\lab{linpunt}
\lim_{h\to 0}\int_{\om}\mathcal V_{h}(\xx,\nabla\v)\,d\xx= \int_{\om}\mathcal V_{0}(\xx,\E(\v))\,d\xx
\end{equation}
Right hand side in \eqref{linpunt} represents the classical linear elastic deformation energy and such a limit was retained to establish a reasonable justification
of linearized elasticity.
However in \cite{DMPN} it is proved by $\Gamma$-convergence techniques that, under standard structural conditions on $\mathcal W$, actually the linear elastic problem is achieved in the limit by exploiting the weak convergence of  $H^1(\om,\R^N)$,
in case of Dirichlet or mixed boundary condition.
\vskip0.2cm
The variational limit is different 
when no Dirichlet boundary condition is present, as we outline briefly here:
in \cite{traction} we studied the case of Neumann boundary conditions,     that is pure traction problem in elasticiy, 
by considering the sequence of energy functionals
\begin{equation}
\displaystyle \F_h(\v)\ =\
\displaystyle \int_\Omega\mathcal V_{h}(\xx, \nabla \v)d\xx-\mathcal L(\v)\,, \qquad \ \v\,\in\, H^1(\om,\R^N)\,,
 \end{equation}
and we we inquired whether  the asymptotic relationship $\F_{h}(\v_{h})=\inf \F_{h}+o(1)$
as $h\to 0_+$ implies, up to subsequences, some kind of
weak convergence of  $\v_{h}$ to a minimizer $\v_{0}$ of a suitable limit functional
in $H^{1}(\om;\mathbf R^{N})$; to this aim next example is highly explicative:
assume
\begin{equation}
\lab{WquadIntrod}
\W(\xx,\mathbf F)=\left\{\begin{array}{ll}
 |\mathbf F^{T}\mathbf F-\Id|^{2}\  &\hbox{if} \ \det \mathbf F>0\\
\vspace{0.1cm}
&\\
 +\infty\ &\hbox{otherwise,}\\
 \end{array}\right.
\end{equation}
$\mathbf g\equiv \mathbf f\equiv \mathbf 0\,$, hence $\inf \mathcal F_{h}= 0$ for every $h>0$,
then by choosing
a fixed nontrivial $N\times N$ skew-symmetric matrix $\mathbf W$, a real number $0< 2\alpha < 1$
and setting
\begin{equation}
\z_{h}:=h^{-\alpha}\,\mathbf W
\,\x \,,
\end{equation}
we get  $\ \mathcal F_{h}(\z_{h})= \inf \mathcal F_{h}+o(1)$, though
$\z_{h}$  has no subsequence
weakly converging in  $H^{1}(\om;\mathbf R^{N})$.
\\
Therefore in contrast to
\cite{DMPN},
one cannot expect weak $H^{1}(\om;\mathbf R^{N})$ compactness of minimizing sequences for pure traction problem,
not even in the simplest case of null external forces: 
we emphasize that in nonlinear elasticity this difficulty cannot be easily circumvented in general
by standard translations since $\F_{h}(\v_{h})\!\not=\! \F_{h}(\v_{h}-\mathbb P\v_{h})$,
with $\mathbb P$ projection on infinitesimal rigid displacements.
Nevertheless, we will show in Theorem~\ref{strongconvmin} below that,
 at least for some special $\mathcal W$, if $ \mathcal F_{h}(\v_{h})= \inf \mathcal F_{h}+o(1)$ then up to subsequences
$\mathcal F_{h}(\v_{h}-\mathbb P\v_{h})= \inf \mathcal F_{h}+o(1)$.
\\
For this reason, we exploited a much weaker topology: the
weak $L^{2}(\om;\mathbf R^{N})$ convergence of linear strains.
Since such convergence  does not imply an analogous convergence of the skew symmetric part of the gradient of displacements, one may expect that the $\Gamma\,$limit functional is different from the  point-wise limit of
$\mathcal F_{h}$, as actually is the case.
\\
Under some natural assumptions on $\mathcal W$, a careful application of the Rigidity Lemma of  \cite{FJM}
together with a suitable tuning of asymptotic analysis with Euler-Rodrigues formula for rotations
show that, if $\mathbb E(\v_{h})$ are bounded in $L^{2}$, then up to subsequences $\sqrt h\,\nabla \v_{h}$ converges strongly in $L^{2}$ to a constant skew symmetric matrix and the variational limit of the sequence $\mathcal F_{h}$, with respect to the w-$L^{2}$ convergence of linear strains, turns out to be
the functional $\mathcal F$ defined in \eqref{DTfuncintro}:
in \cite{traction} it is proved that
if loads are equilibrated and fulfil the compatibility condition
\beeq
\lab{compintro}
\int_{\partial\Omega}\!\!\mathbf f\cdot\mathbf W^{2}\x\, d\H^{N-1}
\!+\!
\int_{\Omega}\!\mathbf g\cdot\mathbf W^{2}\x\, d\xx < 0 \qquad\forall\,
\hbox{skew symmetric matrix } \mathbf W\!\not=\! \mathbf 0
\eneq
then  pure traction problem in linear elasticity is rigorously deduced via $\Gamma$-convergence from
the corresponding pure traction problem formulated in nonlinear elasticity,
referring to weak $L^2$ convergence of the linear strains; moreover
minimizers of $\F$ coincide with the ones of of linearized elasticity functional
$\mathcal E$; thus providing a complete variational justification of  pure traction problems in linear elasticity at least if \eqref{compintro} is satisfied.
In particular, as it is shown in Remark 2.8,  this is true when $\mathbf g\equiv 0,\ \mathbf f=f\n$ with $f>0$ and $\n$ is the outer unit normal vector  to $\partial\om,$ that is when we are in presence of tension-like surface forces.
%
\vskip0.5cm
In the present paper we prove some relevant properties concerning the structure of the new functional and improve its variational connection for a particular but significant class of nonlinear energies.\\
In section \ref{sect further prop of F} we prove that
$\mathcal F$ is sequentially lower semicontinuous weak respect to the natural but very weak notion of convergence, e.g. weak\,$L^2$ of linearized strains
(see Proposition \eqref{Fwlsc}),
though
$\mathcal F$ exhibits a kind of "nonlocal" behavior (see Remark \ref{rmknonlocal}).\\
In the 2D case we can prove that $\F$ is a convex functional for every choice of the positive definite quadratic form $\mathcal V_0$ or,
equivalently, for the variational limit of every nonlinear stored energy $\mathcal W$ fulfilling structural assumptions of general kind in the theory of elasticity: this is shown by making explicit
its first variation and showing that the second variation cannot be negative (see \eqref{firstvarF} and Proposition \ref{convF}).
\\ On the other hand in the 3D case the functional $\F$ cannot be convex for
whatever choice of the positive definite quadratic form $\mathcal V_0$ or, equivalently for every nonlinear stored energy $\mathcal W$
fulfilling the standard structural assumptions: see Proposition \ref{nonconvF} and the general counterexample to convexity therein.
\\
The dichotomy above relies on the fact that there exist pairs of skew-symmetric matrices $\mathbf W_1,\mathbf W_2 \in \mathcal M^{3\times 3}_{skew}$ such that
$\mathbf W_1^2+\mathbf W_2^2$ is not the square of any skew-symmetric matrix: e.g. see \eqref{N=3contrextoconvexity}; while in the 2D
case the matrix $\mathbf W^2$ is a nonpositive multiple of the identity for every skew-symmetric matrix $\mathbf W$.
\\
Notice that $\mathcal F$ is not subadditive: indeed already in dimension $N=2$ formula \eqref{F=E-err} shows that functional $\mathcal F$ cannot be subadditive
on disjoint sets.
\\
In Section \ref{sectionPrelRes}
for reader's convenience we summarize and comment preliminary main results of \cite{traction}
about the variational convergence of pure traction problems.
\\
Eventually, in Section \ref{section strong conv}
we refine the convergence properties for \textsl{minimizing sequences  of the sequence of functionals} $\F_h$ (e.g $\F(\v_{h})=\inf\F_{h}+o(1)$): if $\mathcal W$
is the {\it Green-St.Venant} energy density \eqref{GSV1} then we show by Theorem \ref{strongconvmin} that there exist subsequences
of functionals $\mathcal F_{h}$ and of related minimizing sequence $\v_{h}$, such that (without relabeling) $\v_{h}\!-\!
\mathbb P\v_{h}$ converges weakly in $H^{1}(\om;\mathbf R^{N})$ and strongly in $W^{1,q} (\om,\R^N)\  (1\leq q <2) $ to a minimizer of $\mathcal F$, provided both \eqref{globalequiintro2} and \eqref{compintro} hold true.\\
On the other hand, if inequality in \eqref{compintro} is fulfilled only in a weak sense
by the collection of skew symmetric matrices, then still $\argmin \F$ contains $\argmin \mathcal E $ and  $ \min \mathcal F=\min \mathcal E$,  but $\F$
may have infinitely many minimizing critical points which are not minimizers of $\mathcal E$.
\\
Therefore, only two cases are allowed: either $\min \mathcal F=\min \mathcal E$ or $\inf \F=-\infty$;
actually the second case arises in presence of compressive surface load.
\\
We mention several contributions facing issues in elasticity which are strictly connected with the context of present paper:
\cite{ADMDS},\,\cite{ADMLP}, \cite{ABP},\,\cite{AP}\,\cite{BBGT},\,\cite{BD},\,\cite{BT},\,\cite{CLT},\,\cite{LM}\,\cite{MPT},
\!\cite{MPT2},\,\cite{MPTFvK},\,\cite{PT1},\,\cite{PTplate},\,\cite{PT2},\,\cite{PT4}.
%
%
%

\section{Structural properties of functional $\mathcal F$}
\label{sect further prop of F}
In this section we develop further the analysis of structural properties
of functional $\F$ defined by \eqref{DTfuncintro}, focussing mainly on convexity and semicontinuity issues.\vskip0.1cm
All along the paper we assume that the reference configuration of the elastic body is a
\begin{equation}\label{OMEGA}
\hbox{ bounded, connected open set } \Omega \subset \R^N \hbox{ with Lipschitz boundary, }  \ N= 2, 3,\quad
\end{equation}
and set these notations:
the generic point $\xx\in \Omega$ has components $x_j$ referring to the standard basis vectors
${\mathbf e}_j$ in $\R^N$;
${\mathcal L}^N$ and ${\mathcal B}^N$ denote respectively the
$\sigma\mbox{-algebras}$
of Lebesgue measurable and Borel measurable subsets of $\R^N$.
\vskip0.1cm
The notation for vectors $\aa,\,\bb\in\R^N$ and $N \!\times \!N$ real matrices $\mathbf A,\,\mathbf B,\, \mathbf F $ are as
follows: $\aa\cdot\bb=\sum_j\aa_j\bb_j\,;$
$\mathbf A\cdot\mathbf B=\sum_{i,j}\mathbf A_{i,j}\mathbf B_{i,j}\,;$
$[\mathbf A\mathbf B]_{i,j}=\sum_{k}\mathbf A_{i,k}\mathbf B_{k,j}\,;$
$| \mathbf F|^2 =\hbox{Tr}(\mathbf F^{T}\mathbf F)=\sum_{i,j}F_{i,j}^2$ denotes the squared Euclidean norm of $\mathbf{F}$ in the space $ \M^{N \times N}$ of $N\!\times\! N$ real matrices;
$\mathbf I \in \M^{N \times N}$ denotes the identity matrix,
$SO(N)$ denotes the group of rotation matrices, $\M^{N\times N}_{sym}$ and $\M^{N\times N}_{skew}$ denote respectively the sets of symmetric and skew-symmetric matrices.
For every $\mathbf B\in \M^{N \times N}$ we define ${\rm sym\,}\mathbf B:=\frac{1}{2}(\mathbf B+\mathbf B^T)$
and  ${\rm skew\,}\mathbf B:=\frac{1}{2}(\mathbf B-\mathbf B^T)$.
\vskip0.1cm
First we recall that the minimum at right-hand side in definition \eqref{DTfuncintro} of $\mathcal F$ exists
for every $\vv$ in $H^1(\om,\R^N)$, so that $\mathcal F(\vv)$ is well defined: precisely the finite dimensional minimization problem has exactly two solutions which differs only by the sign, since
strict convexity of the positive definite quadratic form ${\mathcal V_0}(\xx,\cdot)$ entails
\begin{equation}
\displaystyle\lim_{|\mathbf W| \to +\infty,\, \mathbf W\in \M^{N\times N}_{skew}}\int_\Omega
{\mathcal V_0}\left(\xx,\mathbb E(\v)-\textstyle\frac{1}{2}\mathbf W^{2}\right)
\, d\xx\ =\ +\infty\,
\end{equation}
and hence the existence of a unique minimizer $\mathbf W^{2}$.
%
\begin{proposition}\lab{convF} If $N=2$ then functional $\F$ is convex for every choice of the positive definite quadratic form $\mathcal V_0$.
\end{proposition}
\begin{proof}For every $\eps>0$ we define $\varphi_{\eps}\in C^{2}(\mathbb R)$ as
\begin{equation}\lab{phieps}
\varphi_{\eps}(t)=
\left\{
\begin{array}{ll}
t^{2}-\eps t+\frac{\eps^{2}}{3} \ &\hbox{if}\  t\le 0
\vspace{0.1cm}
\\
(3\eps)^{-1}(\eps-t)^{3}\ & \hbox{if}\  0\le t\le \eps
\vspace{0.1cm}
\\
 0\ & \hbox{otherwise}
\end{array}\right.
\end{equation}
and introduce the $C^2$ functionals $\mathcal F_{\eps}$
 by setting
 \begin{equation}\lab{Feps}
\displaystyle \mathcal F_{\eps}(\v)\ =\ \mathcal E(\v) - \,\frac{1}{4}\!\left (\int_{\om}\mathcal V_0(\xx,\mathbf I)
d\xx\right )^{\!\!-1} \!\!\varphi_{\eps}\!\left(\int_{\om}D\mathcal V_0(\xx,\mathbf I)\!\cdot\!\mathbb E(\v)
\,d\xx\right)
\quad\forall\, \v\in H^{1}(\om,\mathbb R^{N})\,.
\end{equation}
Then
by \eqref{phieps}, \eqref{Feps} and representation 
\begin{equation*}
 \mathcal F(\vv) \ = \ \mathcal E(\vv)
-\frac{1}{4}\left (\int_{\om}\mathcal V_0(\xx,\mathbf I)
d\xx\right )^{\!-1}\left [\left (\int_{\om}D\mathcal V_0(\xx,\mathbf I)\!\cdot\!\mathbb E(\v)
\,d\xx\right)^{\!-}\right ]^{2} ,
\end{equation*}
we get $$ \mathcal F_{\eps}\le \F \ ,\qquad \ \F= \sup_{\eps>0} \F_{\eps}\ .$$
Moreover we claim that $\F_{\eps}$ is convex for every $\eps>0$ and this property entails the convexity of $\F$ since $\F$ is the supremum of a family of convex functions. \\
Indeed $\F_{\eps}$ is a $C^{2}$ functional on the whole space $H^{1}(\om,\mathbb R^{N})$ therefore
its second variation,
for every $\uu,\v\in H^{1}(\om,\mathbb R^{N})$, is
\begin{equation}\begin{array}{ll}\lab{secvar}
&\displaystyle\v^{T}\delta^{2}\!\F_{\eps}(\uu)\v= \v^{T}\delta^{2}\mathcal E(\uu)\v\\
&\\
&\displaystyle
- \frac{1}{4}\left (\int_{\om}\mathcal V_0(\xx,\mathbf I)
d\xx\right )^{\!-1}\varphi''_{\eps}\left(\int_{\om}D\mathcal V_0(\xx,\mathbf I)\!\cdot\!\mathbb E(\uu)
\,d\xx\right)\left(\int_{\om}\!
 D \mathcal V_0(\xx,\mathbf I)\!\cdot\!\mathbb E(\v)\right)^{2}=\\
 &\\
 &\displaystyle=2\int_{\om}\mathcal V_{0}(\xx, \mathbb E(\v))\,d\xx-\\
 &\\
 &\displaystyle-\frac{1}{4}\left (\int_{\om}\mathcal V_0(\xx,\mathbf I)
d\xx\right )^{\!-1}\varphi''_{\eps}\left(\int_{\om}D\mathcal V_0(\xx,\mathbf I)\!\cdot\!\mathbb E(\uu)
\,d\xx\right)\left(\int_{\om}\!
 D \mathcal V_0(\xx,\mathbf I)\!\cdot\!\mathbb E(\v)\right)^{2}\,.
 \end{array}
\end{equation}
By taking into account that $0\le \varphi''_{\eps}\le 2$ we get
\begin{equation}\lab{secvarin}
\displaystyle\v^{T}\delta^{2}\F_{\eps}(\uu)\v\ge 2\int_{\om}\mathcal V_{0}(\xx, \mathbb E(\v))\,d\xx-\frac{1}{2}\left (\int_{\om}\mathcal V_0(\xx,\mathbf I)
d\xx\right )^{\!-1}\left(\int_{\om}\!
 D \mathcal V_0(\xx,\mathbf I)\!\cdot\!\mathbb E(\v)\right)^{2}\,.
\end{equation}
Hence, representation \eqref{F=E-err} entails that the right hand side of \eqref{secvarin} is
$2\big(\F(\v)+\mathcal L(\v)\big)$ if
$$\int_{\om}\!
 D \mathcal V_0(\xx,\mathbf I)\!\cdot\!\mathbb E(\v)\ge 0$$
and is $\F(-\v)+\mathcal L(-\v)$ else. Therefore in both cases  \eqref{DTfuncintro} entails $\v^{T}\delta^{2}\F_{\eps}(\uu)\v\ge 0$ for every $\uu,\,\v\in H^{1}(\om,\mathbb R^{N})$. Therefore $\F_{\eps}$ is convex and claim is proved.
\end{proof}
\begin{proposition}\lab{nonconvF} If $N=3$ then functional $\F$ is nonconvex for every choice of the positive definite quadratic form $\mathcal V_0$.
\end{proposition}
\begin{proof}
Set
\begin{equation}
\label{N=3contrextoconvexity}
\mathbf W_{1}=\e_{1}\otimes\e_{2}-\e_{2}\otimes\e_{1},\ \mathbf W_{2}=\e_{2}\otimes\e_{3}-\e_{3}\otimes\e_{2}.
\end{equation}
Then
$$ \frac{1}{2}(\mathbf W_{1}^{2}+\mathbf W_{2}^{2})=-\frac{1}{2}(\e_{1}\otimes\e_{1}+\e_{3}\otimes\e_{3})-\e_{2}\otimes\e_{2}:= \mathbf A$$
and by choosing
$\v(\xx):= \mathbf A\xx$ we get  $\mathbb E(\v)=\mathbf A\not\in\{\mathbf W^{2}: \mathbf W\in \M^{N\times N}_{skew}\}$. \\ Hence, due to \eqref{Z3}, $\F(\v)>-\mathcal L(\v)$ for every  choice of $\mathcal V_{0}$. By setting
  $$\v_{1}(\xx):=\mathbf W_{1}^{2}\xx, \ \v_{2}(\xx):=\mathbf W_{2}^{2}\xx$$
   we get  $\F(\v_{1})=-\mathcal L(\v_{1}),\ \F(\v_{2})=-\mathcal L(\v_{2})$ hence
$$\textstyle\F(\frac{1}{2}(\v_{1}+\v_{2}))=\F(\v)>-\mathcal L(\v)=-\frac{1}{2}(\mathcal L(\v_{1})+\mathcal L(\v_{2}))=\frac{1}{2}(\F(\v_{1})+\F(\v_{2}))$$
thus proving that $\F$ is not convex in the 3D case for every choice of $\mathcal V_{0}$.
\end{proof}
Although existence of minimizers of $\F$ is already a direct consequence of convergence results in [23],
in the next Proposition we provide a direct proof of
sequential lower semicontinuity of $\F$ with respect to the natural, very weak convergence,
for both cases of dimension $2$ and $3$.
\begin{proposition} \label{Fwlsc}
Assume that the \textsl{standard structural conditions} and
\eqref{globalequiintro2} holds true.\\ Then
for every $\v_{n},\v\in H^1(\Omega;\R^N)$ such that
$\mathbb E(\v_n)\wconv\mathbb E(\v)$ in $L^2(\Omega;\M^{N\times N})$ we have
$$\liminf_{n\rightarrow +\infty} \mathcal F(\v_n)\ \geq \  \mathcal F(\v)$$
\end{proposition}
\begin{proof} Let  $\v_{n},\v$ belong to $ H^1(\Omega;\R^N)$ and fulfil
$\mathbb E(\v_n)\!\wconv\!\mathbb E(\v)$ in $L^2(\Omega;\M^{N\times N})$. Then
$\mathbb E(\v_n)$ is bounded in $L^2(\Omega;\M^{N\times N})$. If $\liminf_{n\rightarrow +\infty} \mathcal F(\v_n)= +\infty$ then the claim is trivial, so we may also assume without restriction that  $\F(\v_{n})\le C$.
Assumption
\eqref{globalequiintro2} of equilibrated load entails $\mathcal F(\v_{n})=\mathcal F(\v_{n}-\mathbb P\v_{n})$,
so may suppose that $\mathbb P\v_{n}\equiv \mathbf 0$.
We choose
\begin{equation}\label{Wn}
\displaystyle\mathbf W_{n}\in \argmin \left\{\int_{\om}\mathcal V_0\Big( \xx,\, \mathbb E(\v_{n})-\textstyle\frac{1}{2}\mathbf W^{2}\Big)\, d\xx: \ \mathbf W\in \M^{N\times N}_{skew}\right\} \,.
\end{equation}

hence, 
if $C_K$ the Korn-Poincar\'e inequality in $\Omega$
and $\alpha>0$ is the uniform coercivity constant of $\mathcal V_0$, say $\mathcal V_0(\xx,\mathbf M)\geq \alpha|\mathbf M|^2$,
we get
\begin{equation}
\begin{array}{lll}\displaystyle
\alpha\int_{\om}|\mathbb E(\v_{n})-\textstyle \frac{1}{2}\mathbf W_{n}^{2}|^{2}\,d\xx
&\leq&
C+\mathcal L(\vv_n)\,=\,
C+\mathcal L(\v_n-\mathbb P\vv_n)\,\leq\, \vspace{0.1cm}\\
&\leq&
C+C_K(\|\mathbf f\|_{L^{2}(\partial\om)}+\|\mathbf g\|_{L^{2}(\om)})\|\mathbb E(\vv)\|_{L^2(\Omega;\M^{N\times N})}
\,,
\end{array}
\end{equation}
Therefore $|\mathbf W_{n}^2|$ is bounded and since $\mathbf W_{n}$ is real skew-symmetric
we obtain that
$|\mathbf W_{n}|$ is bounded too. So we may suppose that, up to subsequences, $\mathbf W_{n}\to \mathbf W$ in $ \M^{N \times N}_{skew}$. By taking into account that $\mathbb P\v_{n}\equiv \mathbf 0$ we get $\v_{n}\wconv \v$ in $H^{1}(\om, \mathbb R^{N})$ hence by recalling that $\mathcal V_{0}(\xx, \cdot)$ is a convex quadratic form
\begin{equation}\begin{array}{ll}
\displaystyle \liminf_{n\rightarrow +\infty} \mathcal F(\v_n)\! &=\ \,\liminf_{n\rightarrow +\infty}\int_{\om}\mathcal V_0\Big( \xx,\, \mathbb E(\v_{n})-\textstyle\frac{1}{2}\mathbf W^{2}_{n}\Big)\, d\xx- \mathcal L(\v_{n})\ge
\vspace{0.3cm}
\\
&\displaystyle \ge\ \int_{\om}\mathcal V_0\Big( \xx,\, \mathbb E(\v)-\textstyle\frac{1}{2}\mathbf W^{2}\Big)\, d\xx- \mathcal L(\v)\ge \mathcal F(\v)\\
\end{array}
\end{equation}
which proves the claimed lower semicontinuity inequality.
\end{proof}
\begin{remark}\label{variat}
{\rm
The first variation of $\F$ can be explicitly evaluated in the 2D case, thanks to \eqref{F=E-err}, as follows
\begin{equation}\label{firstvarF}
\begin{array}{cl}
\delta\mathcal F(\v)[\varphi] 
\displaystyle
=&\displaystyle
\!\!\!\!
 \int_{\om}\!
D\mathcal V_0\big(\xx,\mathbb E(\v)\big)\!\cdot\mathbb E(\varphi)
\,d\xx\,
\\
&\displaystyle
\hskip-1.1cm
+
\frac{1}{2}\!
\left (\int_{\om}\!\mathcal V_0(\mathbf I)d\xx\right )^{\!\!-1}\!\!\left (\int_{\om}
    D\mathcal V_0(\xx,\mathbf I)\!\cdot\!\mathbb E(\v)
\,d\xx\right)^{\!\!-} \!\!\!
\int_{\om}\!
   D \mathcal V_0(\xx,\mathbf I)\!\cdot\!\mathbb E(\varphi)
\,d\xx -\mathcal L(\varphi)=
\\
&\displaystyle
\hskip-1.5cm
=\,
\delta\mathcal E(\v)[\varphi]
\,+\,
\frac{1}{2}\!
\left (\int_{\om}\!\mathcal V_0(\mathbf I)d\xx\right )^{\!\!-1}\!\!\left (\int_{\om}
    D\mathcal V_0(\xx,\mathbf I)\!\cdot\!\mathbb E(\v)
\,d\xx\right)^{\!\!-} \!\!\!
\int_{\om}\!
   D \mathcal V_0(\xx,\mathbf I)\!\cdot\!\mathbb E(\varphi)
\,d\xx
\end{array}
\end{equation}
for every $\vv,\,\varphi\in H^{1}(\om;\mathbf R^{N})$.
}
\end{remark}
\begin{remark}\label{rmknonlocal}{\rm
Functional $\F$ exhibits a nonlocal behavior: precisely in 2D, due to the representations \eqref{F=E-err}
and \eqref{firstvarF}  respectively of the functional  of first variation, $\F(\v)$ is the sum
of a contribution $\mathcal E(\v)$ due to local functional $\mathcal E$ related to linear elasticity plus a possibly vanishing contribution with global dependance on $\v$
explicitly evaluated by
$$
 -\,\frac {a_*(\v)}4 \ =\
-\frac{1}{4}\left (\int_{\om}\mathcal V_0(\xx,\mathbf I)
d\xx\right )^{\!-1}\left [\left (\int_{\om}D\mathcal V_0(\xx,\mathbf I)\!\cdot\!\mathbb E(\v)
\,d\xx\right)^{\!-}\right ]^{2}\,,
$$
which simplifies as follows in the case of Green-Saint Venant energy:
$$ -\,\frac 1 {4\,|{\om}|}\left (\int_{\om}\hbox{\rm div}\,\v\,d\xx\right)^{\!-}\ ,$$
while the nonlocal coefficient $ \left (\int_{\om}
    D\mathcal V_0(\xx,\mathbf I)\!\cdot\!\mathbb E(\v)
\,d\xx\right)^{\!\!-} $ appears in Euler equations.}
\end{remark}


\section{Preliminary variational convergence results}
\label{sectionPrelRes}
In this Section 
we recall the main results of \cite{traction}
about the variational convergence of pure traction problems.
To this aim
basic notation and assumptions for general nonlinear energies is introduced first.
\\
Still we assume that the reference configuration of the elastic body is a
\begin{equation}\label{OMEGA}
\hbox{ bounded, connected open set } \Omega \subset \R^N \hbox{ with Lipschitz boundary, }  \ N= 2, 3,\quad
\end{equation}
and set these notations:
the generic point $\xx\in \Omega$ has components $x_j$ referring to the standard basis vectors
${\mathbf e}_j$ in $\R^N$;
${\mathcal L}^N$ and ${\mathcal B}^N$ denote respectively the
$\sigma\mbox{-algebras}$ of Lebesgue measurable and Borel measurable subsets of $\R^N$.\\
For every $\mathcal U:\om\times \M^{N \times N}\rightarrow \mathbb R,$ with $ \mathcal U(\xx,\cdot)\in C^{2}$ a.e. $\xx\in \om$,
we denote the gradient and the hessian of $g$ with respect to the second variable
by $D\mathcal U(\xx,\cdot)$ and $D^{2}\mathcal U(\xx,\cdot)$ respectively.
\\
For every displacements field $\vv\in H^1(\Omega;\R^N)$,  $\E(\vv):= \hbox{sym }   \!\nabla\vv$ denotes the infinitesimal strain tensor field,  $\mathcal R:=\{ \vv\in H^1(\Omega;\R^N): \E(\v)=\mathbf 0\}$ the set of infinitesimal rigid displacements and $\mathbb P\vv$ is the orthogonal projection of $\vv$ onto $\mathcal R$.
\vskip0.1cm
We consider a body made of an hyperelastic material, say
there exists a ${\mathcal L}^N\! \!\times\! {\mathcal B}^{N^2} $measurable
$\W : \om \times \M^{N \times N} \to [0, +\infty ]$
such that,  for a.e. $\xx \in \om$,
$\W(\xx,\nabla \mathbf y(\xx))$ represents the stored energy density, when $\mathbf y(x)$ is the deformation and $\nabla \mathbf y(\xx)$ is the deformation gradient.\\
Moreover we assume that for a.e. $\xx \in \om$
\beeq\lab{incom}  
	\W(\xx,\mathbf F)=+\infty \qquad \mbox{if $\det \mathbf F \leq 0$} \quad\hbox{(orientation preserving condition)}\,,
\eneq
\beeq \lab{framind} \W(\xx, \mathbf R\mathbf F)=\W(\xx, \mathbf F)\qquad \forall \, \mathbf R\!\in\! SO(N) \quad
\forall\, \mathbf F\in \M^{N \times N}\ \quad  \hbox{(frame indifference)}\,,
\eneq
\beeq\lab{reg} \exists\   \hbox{a neighborhood} \ \mathcal A  \  \hbox{of}\ SO(N) \hbox{ s.t.}
\quad\mathcal W(\xx,\cdot)\in C^{2}(\mathcal A)\,,
\eneq
\beeq \lab{coerc}
\exists\, C\!>\!0 \hbox{ independent of }\xx:\ 
\W(\xx,\mathbf F)\ge  C|\mathbf F^{T}\mathbf F-\mathbf I|^{2}\ \,\forall\, \mathbf F\!\in\! \M^{N \times N}\,   \hbox{(coerciveness)},
\eneq
\beeq \lab{Z1}
	\W(\xx,\mathbf I)=0\,, \quad D \W(\xx,\mathbf I)=0\,, \qquad \hbox{for a.e. }\xx \in \om  \,,
\eneq
that is  the reference configuration
has zero energy and
is stress free, so
by \eqref{framind}  we get also
$$\W(\xx,\mathbf R)\!=\!0,\quad D \W(\xx,\mathbf R)\!=\!0 \qquad \ \forall \,\mathbf R\in SO(N)\,.$$
By frame indifference there exists a  ${\mathcal L}^N \!\times\! {\mathcal B}^N \mbox{-measurable}$ $\V : \om \times \M^{N \times N} \to [0,+\infty]$ such that for every $\mathbf F\in \M^{N\times N}$
\beeq
\lab{WhatZ}
\W(\xx,\mathbf F)=\V(\xx,\textstyle{\frac{1}{2}}( \mathbf F^T \mathbf F - \mathbf I))
\eneq
and by \eqref{reg}
%
\beeq\lab{regV} \exists\   \hbox{a neighborhood} \ \mathcal O \  \hbox{of}\ \mathbf 0 \hbox{ such that}
\ \mathcal V(\xx,\cdot)\in C^{2}(\mathcal O), \hbox{ a.e. }x \in \om \,.
\eneq
In addition we assume that  there exists $\gamma>0$ independent of $\xx$ such that
\beeq \lab{Velliptic}
	\left|\,\mathbf B^T\,D^2 \mathcal V(\xx,\mathbf D) \,\mathbf B\,\right| \, \leq \, 2\, \gamma\, |\mathbf B|^2
	\quad
	\forall\, \mathbf D\!\in\! \mathcal O, \ \forall\, \mathbf B \!\in\! \M^{N \times N} .
\eneq


By (\ref{Z1}) and Taylor expansion with Lagrange reminder we get,
for a.e. $\xx \in \om$ and suitable $t \in (0,1)$ depending on $\xx$ and on $\mathbf B$:
\beeq \lab{12'}
	\mathcal V(\xx,\mathbf  B)\ =\ \frac{1}{2}\, \mathbf  B^{T} D^2\V (\xx, t \mathbf  B)\, \mathbf  B\,.
\eneq
Hence  by \eqref{Velliptic}
\beeq \lab{upbound}
	\mathcal V(\xx,\mathbf  B) \, \leq \, \gamma \,|\mathbf  B|^2
	\qquad \forall \ \mathbf  B\in \M^{N \times N}\cap  \mathcal O\,.
\eneq
According to \eqref{WhatZ}
for a.e.  $ \xx \!\in\! \om,\ h\!>\!0$ and every $ \mathbf  B\in \M^{N \times N} $ we set
\beeq \lab{Wh}
	\mathcal V_{h} (\xx,\mathbf B) := \frac{1}{h^2} \, \mathcal W(\xx,\Id+h \mathbf B) =
	\frac{1}{h^2} \, \mathcal V(\xx,h \, \hbox{\rm sym\,} \mathbf B +
	\textstyle{\frac{1}{2}}h^2 \mathbf B^T \mathbf B)\ .
\eneq
Taylor's formula with \eqref{Z1},\eqref{Wh} entails
$\mathcal V_h(\xx, \mathbf B)= \frac 1 2 \, (\hbox{\rm sym\,} \mathbf B)\,
D^2\mathcal V(\xx,\mathbf 0) \,
 (\hbox{\rm sym\,} \mathbf B) + o(1)$, so
\begin{equation}\lab{Vh}
 \mathcal V_h(\xx, \mathbf B)\ \to\ \mathcal V_0 (\xx, \hbox{\rm sym } \mathbf B) \hbox{ as } h\to 0_+\,,
 \end{equation}
where the point-wise limit of integrands is the quadratic form $\mathcal V_0$ defined by
 \begin{equation}\label{A e Q}
    \mathcal V_0 (\xx,\mathbf B)
    \ :=\ \frac 1 2 \,\mathbf B^TD^2\mathcal V(\xx,\mathbf 0) 
  \,\mathbf B\,
  \quad \hbox{a.e. }\,\xx\in\Omega,\ \mathbf B\in \M^{N\times N}\,.
  \end{equation}
The symmetric fourth order tensor $D^2\mathcal V(\xx,\mathbf 0)$ in \eqref{A e Q} plays the role of the classical elasticity tensor.
\\
By 
\eqref{coerc} we get
\begin{equation}\label{(2.18)}
  \mathcal V_{h} (x,\mathbf B)\ =\ \frac{1}{h^2} \, \mathcal W(x,\Id+h \mathbf B)\ \ge\
  C\,|\,2\,{\rm sym } \mathbf B\,+\,h\,\mathbf B^T \mathbf B\,|^{2}
\end{equation}
so that \eqref{A e Q} and \eqref{(2.18)} imply the ellipticity of $\mathcal V_0$\,:
\beeq \lab{Z3}
	\mathcal V_0 (\xx,\hbox{\rm sym}\,\mathbf B)
	\, \geq \, 4\,C \, |\hbox{\rm sym}\,\mathbf  B|^2
	\qquad 
\ \hbox{a.e. }\,\xx\in\Omega,\ \mathbf B\in \M^{N\times N}\,.
\eneq
%
For a suitable choice of the adimensional parameter $h>0$, the functional representing the total energy is labeled by $\mathcal F_h: H^1(\Omega;\R^N)\to \R\cup\{+\infty\} $ and defined as follows
\begin{equation}
\label{nonlinear}
\displaystyle \mathcal F_h(\v)\ :=\
\int_\Omega \mathcal V_{h}(\xx, \nabla\v)\, d\xx\,-\,\mathcal L(\v)\,,
\end{equation}
where $\mathcal L$ is defined by \eqref{globalequiintro1}.\\
%
In order to describe the  asymptotic behavior as $h\to 0_+$
of functionals $\mathcal F_h$,
we
refer to the limit energy functional $\mathcal F: H^1(\Omega;\R^N)\to \R$ defined by \eqref{DTfuncintro}.
\\ In this Section we assume
\eqref{OMEGA} 
together with the \textsl{standard structural conditions} \eqref{incom}-\eqref{Z1},\eqref{Velliptic} 
as usual in scientific literature concerning elasticity theory
and we refer to the notations
\eqref{WhatZ},\eqref{Wh},\eqref{A e Q},\eqref{nonlinear}.
\begin{definition}\label{minimseq}
Given an infinitesimal sequence $h_j$  of positive real numbers, we say that $\v_j\in H^1(\Omega;\R^N)$ is a \textsl{minimizing sequence} of the sequence of functionals $\mathcal F_{h_{j}}$
if $$(\mathcal F_{h_{j}}(\vv_j)-\inf\mathcal F_{h_j})\to 0 \quad \hbox{ as } \quad h_j\to 0_+ \,.$$\end{definition}
We proved  
that for every given infinitesimal
sequence $h_j$ actually the minimizing sequences of the sequence of functionals $\mathcal F_{h_j}$
exists. For reader's convenience
we recall here the main results of \cite{traction}: see Lemma 3.1, Theorem 2.2, Theorem 4.1 and Corollary 4.2 therein.
\begin{lemma}\lab{infFh} Assume
the \textsl{standard structural conditions} together with 
\eqref{globalequiintro2} and 
\eqref{compintro}.
\\ Then there is a constant $K$, dependent only on $\Omega$ and the coercivity constant of of the stored energy density appearing in \eqref{coerc}, such that
\begin{equation}\label{infFhdisplay}
  \displaystyle\inf_{h>0}\,\inf_{\vv\in H^1} \ \F_{h}(\vv)\ > \ -\, K\,\big(\|\mathbf f\|_{L^2}^2+\|\mathbf g\|_{L^2}^2\big)\,.
\end{equation}
\end{lemma}
\begin{theorem}
\label{mainth1}
Assume that
the \textsl{standard structural conditions} and
\eqref{globalequiintro2},
\eqref{compintro}
hold true.
Then:
\beeq\lab{equalmin}
 \min_{\v\in H^1(\Omega;\R^N)} \mathcal F(\v)\ =\ \min_{\w\in H^1(\Omega;\R^N)}\mathcal E(\w) \,,
\eneq
\begin{equation}
\label{equivmin}
\argmin_{\v\in H^1(\Omega;\R^N)}  \mathcal F\ =\ \argmin_{\v\in H^1(\Omega;\R^N)}{\mathcal E}\,;
\end{equation}
for every sequence of strictly positive real numbers $h_{j}\to 0$
there are
\textsl{minimizing sequences} of the sequence of functionals $\mathcal F_{h_{j}}$;\\
for every minimizing sequence $\v_j\in H^1(\Omega;\R^N)$
of $\mathcal F_{h_{j}}$ there exist
a subsequence and a displacement
$\v_0\in H^1(\Omega;\R^N)$
 such that, without relabeling,
\begin{equation}
\label{wstr}
\mathbb E(\v_j)\ \wconv\ \mathbb E(\v_{0})\qquad
 weakly \ in \ L^2(\Omega;\M^{N\times N})\,,\quad
\end{equation}
\begin{equation}\label{convsqrth}
\sqrt{h_{j}}\ \nabla\v_j\ \rightarrow\
\mathbf 0
\qquad strongly\ in \ L^2(\Omega;\M^{N \times N})\,,
\end{equation}
\begin{equation}
\label{min}
\lim_{j\to +\infty}\mathcal F_{h_{j}}(\v_j)\ \ =\
\min_{\v\in H^1(\Omega;\R^N)} \! \mathcal F(\v)\ \ =\ \mathcal F(\v_0)\ \ =\ \mathcal E(\v_0)
\,.\qquad
\end{equation}
%
%
\end{theorem}
If strong inequality in the compatibility condition 
\eqref{compintro}
is replaced by a weak inequality, then the uniform estimate \eqref{infFhdisplay} still hold true and also minimizing sequences of the sequence of functionals $\mathcal F_{h_j}$ exist for every infinitesimal
sequence $h_j$, but the minimizers coincidence \eqref{equivmin} for $\mathcal F $ and $\mathcal E$ cannot hold anymore. Nevertheless the following general result holds true.
\begin{proposition} \lab{exinfmin}If the structural assumptions together with 
\eqref{globalequiintro2} are fulfilled, but 
\eqref{compintro}
is replaced by
\begin{equation}
\lab{wcompbis} \mathcal L( \mathbf W ^2 \xx
)\le 0\qquad \forall\, \mathbf W\in \M^{N\times N}_{skew}
\end{equation}
then $\argmin \mathcal F$ is still nonempty and
\beeq\lab{eqmin}
\min\F=\min\mathcal E\,,
\eneq
but the coincidence of minimizers sets is replaced by the inclusion
\begin{equation}\lab{incl}
\argmin\mathcal E\subset\argmin\F\,.
\end{equation}
If \eqref{wcompbis} holds true and there exists $\mathbf U\in \M^{N\times N}_{skew},\ \mathbf U\neq \mathbf 0$ such that $\mathcal L(\mathbf U^2 \xx)=0$,
 then $\F$ admits infinitely many minimizers which are not minimizers of $\mathcal E$, precisely
\begin{equation}\label{argmin=}
\argmin \mathcal E\ \,\mathop{\subset}_{\neq} \ \, \argmin \mathcal E \, +\,  \left\{ \,\mathbf U^2\xx\, : \  \mathbf U \in \mathcal M^{N\times N}_{skew},\ \mathcal L(\mathbf U^2\xx)=0 \right\}
\ \,\mathop{\subset} \ \,\argmin \mathcal F \,,
\end{equation}
where the last inclusion is an equality in 2D:
\begin{equation}\label{argmin==}
\argmin \mathcal E\ \mathop{\subset}_{\neq} \ \argmin \mathcal E \, +\,  \left\{ \,-\, t\,\xx\, : \  t\geq 0 \right\}
\ = \ \argmin \mathcal F \,,\qquad \hbox{ if } N=2\,.
\end{equation}
\end{proposition}
\begin{remark}\lab{unbound}
{\rm The compatibility condition 
\eqref{compintro}
cannot be dropped in Theorem \ref{mainth1}
even if the (necessary) condition 
\eqref{globalequiintro2} holds true. Moreover plain substitution of strong with weak inequality
in
\eqref{compintro}
leads to a lack of compactness for minimizing sequences.
\vskip0.1cm
Indeed, if $\n$ denotes the outer unit normal vector to $\partial\om$ and we choose
$\mathbf f=f\n$ with $f<0$, $\mathbf g\equiv 0$ then
\beeq
\int_{\partial\Omega}\mathbf f\cdot\mathbf W^{2}\,\x\,d\H^{N-1}= 2f(\Tr\mathbf W^{2})\vert\Omega\vert >0
\eneq
and the strict inequality in
\eqref{compintro}
is reversed in a strong sense by any $\mathbf W\!\in\! \M^{N\times N}_{skew}\setminus \{\mathbf 0\}$;\\
fix a sequence of positive real numbers such that $h_{j}\!\to\! 0$ 
$\mathbf W\!\in\! \M^{N\times N}_{skew},\  \mathbf W\not\equiv \mathbf 0$,
and set $\v_j={h_{j}}^{-1}(\frac{1}{2}\mathbf W^2+\frac{\sqrt 3}{2}\mathbf W)\,\xx\,$,
then $\mathbf I +\big(\frac 1 2 \mathbf W^2+\frac {\sqrt 3}{2}\mathbf W\big)\in SO(N) $ and
\begin{equation}
\mathcal F_{h_{j}}(\v_j)=-\frac{f}{2h_{j}}\int_{\partial \Omega}\mathbf W^{2}\x\cdot\n\,d\H^{n-1}=-\frac{f}{2h_{j}}(\Tr\mathbf W^{2})|\Omega|\rightarrow -\infty.
\end{equation}
On the other hand, assume \eqref{OMEGA},
$\W$ as in \eqref{WquadIntrod} and
$\mathbf f\!=\!\mathbf g\!\equiv\! \mathbf 0$, so that the compatibility inequality is susbstituted by the weak inequality;
%
if $\v_j$ are defined as above then,
hence by frame indifference,
 \begin{equation}
\F_{h_{j}}(\v_j)=0=\inf \F_{h_{j}}
\end{equation}
but $\mathbb E(v_j)$ has no weakly convergent subsequences in $L^2(\Omega;\M^{N\times N})$.
}
\end{remark}
\begin{remark} \label{tensatbdry}\rm It is worth noticing that the compatibility condition 
\eqref{compintro}
holds true
when $\mathbf g\equiv 0$, $\mathbf f=f\n$ with $f>0$ and $\n$ the outer unit normal vector to $\partial\om$.\\
Indeed let $\mathbf W\in \M^{N\times N}_{skew}, \mathbf W\not\equiv \mathbf 0$:
hence by 
\eqref{globalequiintro2} and the Divergence Theorem we get
\beeq
\int_{\partial\Omega}\mathbf f\cdot\mathbf W^{2}\,\x\,d\H^{N-1}= 2f(\Tr\mathbf W^{2})\vert\Omega\vert < 0
\eneq
thus proving 
\eqref{compintro}
 in this case.
%
This means that in presence of tension-like surface forces and of null body forces the compatibility condition holds true.
\end{remark}
It is quite natural to ask whether condition 
\eqref{compintro}, which was essential in the proof of Theorem \ref{mainth1}, may be dropped in order to obtain at least existence of $\min \mathcal F$: the answer is negative.\\
Indeed the next remark shows that, when compatibility inequality in 
\eqref{compintro} is reversed
for at least one choice of the skew-symmetric matrix $\mathbf W$, then $\mathcal F$ is unbounded from below.
\begin{remark}\lab{controsegno} If
{\rm 
\beeq\lab{controsegnoeq}
\exists \,\mathbf W_*\in  \M^{N\times N}_{skew}\, :\qquad
\mathcal L(\z_{\mathbf W_*})>0\,, \quad \hbox{where }\z_{\mathbf W_*}=\frac  1 2 \mathbf W_*^2 \xx\,,
\eneq
 then\vskip-0.3cm
\beeq\lab{menoinf}
\inf_{\v\in H^1(\Omega;\R^N)} \mathcal F(\v)\ =\ -\infty.
\eneq
Indeed
we get
\begin{equation}\label{minFvsminF}
  \inf_{ H^1(\Omega;\R^N)} \mathcal F \ =\ \min_{H^1(\Omega;\R^N)} \mathcal E \ \,-
  \sup_{\mathbf W\in \M^{N\times N}_{skew}}
\mathcal L(\z_{\mathbf W}) \qquad \hbox{where }\ \z_{\mathbf W}=\frac 1 2  \mathbf W^2\x\,.
\end{equation}
Hence\vskip-0.3cm
\begin{equation*}
  \inf_{ H^1(\Omega;\R^N)} \mathcal F
  \ \leq \ \min_{H^1(\Omega;\R^N)} \mathcal  E \ \,-  \, \tau\mathcal L(\,\z_{\mathbf W_*}) \qquad \forall \,\tau>0\,,
\end{equation*}
which entails \eqref{menoinf}.}
\end{remark}%
Next example shows that in case of uniform compression along the whole boundary functional $\mathcal F$ is unbounded from below, regardless of convexity or nonconvexity of $\Omega$ and $\mathcal F$.
\begin{example} \lab{nocrit}
{\rm  Assume $\om\subset \R^N$ is a Lipschitz, connected open set$,\  N=2,3,\ \mathbf g\equiv \mathbf 0,\ \mathbf f= -\mathbf n,\ $ where $\n$ denotes the outer unit normal vector  to $\partial\om$.\\
Then \eqref{controsegnoeq} holds true hence, by Remark \ref{controsegno},
$\inf_{\v\in H^1(\Omega;\R^N)} \mathcal F(\v)\ =\ -\infty$. \vskip0.1cm Indeed, for every $ \mathbf W\in
\mathcal M^{N\times N}_{skew}$ such that $|\mathbf W|^2=2$ we obtain}
\\
$$
\int_{\partial \Omega} \!\mathbf f \cdot \mathbf W^2 \xx\, d\mathcal H ^{N-1}\, = \,
- \! \int_{\partial \Omega} \!\mathbf n \cdot \mathbf W^2 \xx \, d\mathcal H ^{N-1}\,= \,
- \!\int_{ \Omega} \dv  (\mathbf W^2 \xx) \,d\xx\,=\, -\, |\Omega|\,\tr\,  \mathbf W^2\,=\, 2\,|\Omega|>0\,.
$$
\end{example}
Summarizing, only two cases are allowed: either $\min \mathcal F=\min \mathcal E$ or $\inf \F=-\infty$:
the second case actually arises in presence of compressive surface load.
\\
The new functional $\mathcal F$ somehow preserves memory of instabilities which are typical of finite elasticity, while they disappear in the linearized model described by $\mathcal E$.
In the light of Theorem \ref{mainth1},
as far as pure traction problems are considered, it seems reasonable that the range of validity of linear elasticity should be restricted to a certain class
of external loads, explicitly those verifying \eqref{compintro}: a remarkable example in such class is a uniform normal tension load at the boundary as in Remark \eqref{tensatbdry};
while
in the other cases equilibria of a linearly elastic body could be better described through critical points of $\F$, whose existence in general seems to be an interesting and open problem.

\section{Strong convergence of minimizing sequences of $\mathcal F_h$}\label{section strong conv}
In this section we prove that for the special class of Green-Saint Venant energy density it is possible to choose a subsequence of functionals $\mathcal F_h$ defined by \eqref{nonlinear} and a corresponding minimizing sequence, according to Definition \eqref{minimseq}, which is weakly converging in $H^{1}(\om;\mathbf R^{N})$  to a minimizer
of $\mathcal F$ defined by \eqref{DTfuncintro}.
Moreover, thanks to a result of \cite{DMPN}, this convergence entails strong convergence in $W^{1,q} (\om;\R^N)$ for
$1\leq q <2$.\\
Before stating the main result of this section we notice that, by frame indifference \eqref{framind} and equilibrated load condition \eqref{globalequiintro2}, without loss of of generality we can assume
\begin{equation}\lab{bar1}
\int_{\om}x_{i}\,d\xx\ =\ 0 \quad\, \forall\ i=1 \ldots N
\,,\qquad
\int_{\om}x_{i}\,x_{j}\,d\xx\ =\ 0 \quad \forall\ i, j=1 \ldots N,\ i\not= j.\\
\end{equation}
Therefore, if $\displaystyle I_{k}$
denotes the moment of inertia of $\om$ with respect to the $k$-th axis, by \eqref{bar1} we get
\begin{equation}\lab{proj}\mathbb P\,\u(\x)\,=\,{\bf a}\times \x,\quad
\displaystyle {\bf a}_{k}\,=\,I_{k}^{-1}\int_{\om}(\x\times\vv)_{k}\,d\xx
\end{equation}
so
\begin{equation}\label{nablaP}
(\nabla\,\mathbb P\,\u(\x))_{k}\ =\ {\bf a}\times {\bf e}_{k}.
\end{equation}

\begin{theorem}
\lab{strongconvmin}
 Let $\mu,\ \lambda>0$,
\begin{equation}\label{(6.1)}
\W(\xx,\mathbf F) \,=\, \W(\mathbf F)\,:=\
\left\{\begin{array}{ll}
 \mu |\mathbf F^{T}\mathbf F-\mathbf I|^{2}+\frac{\lambda}{2}|\ \hbox{\rm Tr}\ (\mathbf F^{T}\mathbf F-\mathbf I)|^{2}\  &\hbox{ if} \ \det \mathbf F>0,\\
 &\\
 +\infty\ &\hbox{ else,}\\
 \end{array}\right.
\end{equation}
assume  
\eqref{globalequiintro2}, 
\eqref{compintro}
and let let $h_{j}$ be a
sequence of strictly positive real numbers with $h_{j}\to 0$.\\ 
Then there exists a (not relabeled) subsequence of functionals $\mathcal F_{h_{j}}$ and a minimizing sequence $\w_j$   weakly converging  in $H^{1}(\om;\mathbf R^{N})$ and strongly converging
in $W^{1,q} (\om,\R^N)$ to $\w_{0}\in \argmin \mathcal E$, for $1\leq q <2$.
\end{theorem}
\begin{proof} By recalling Proposition 5.3 of \cite{DMPN} it will be enough  to  show that there exists
a minimizing sequence $\w_j$ for functionals $\mathcal F_{h_{j}}$
(say $\mathcal F_{h_{j}}(\w_j)=
\inf \mathcal F_{h_{j}}  + o(1)$) weakly converging in $H^{1}(\om;\mathbf R^{N})$ to $\w_{0}\in \argmin \mathcal F$ and
\begin{equation}\lab{energconv}
\displaystyle\lim_{{h_{j}}\to 0}\mathcal F_{h_{j}}(\v_j)= \, \displaystyle \int_{\om} \mathcal V_{0}\big(\mathbb E(\v_0)\big)\,d\xx-\mathcal L(\v_{0})= \,
\mathcal E(\v_{0})
\end{equation}
where it is worth noticing that due to \eqref{(6.1)}
\begin{equation}
\mathcal V_{0}(\xx,\mathbf B)\equiv \mathcal V_{0}(\mathbf B)=4\mu|\mathbf B|^{2}+2\lambda|\ \hbox{Tr}\ \mathbf B|^{2}.
\end{equation}
To this aim let $\v_j$ be a minimizing sequence for functionals $\mathcal F_{h_{j}}$: by Theorem~\ref{mainth1}
 there exist a (not relabeled) subsequence $\{h_{j}\}$ and $\v_j,\  \v_0\in H^1(\Omega;\R^N)$ such that
\begin{equation}
\label{wstr2}
\mathbb E(\v_j)\ \wconv\ \mathbb E(\v_{0})\:\:
\hbox{\rm in }\: L^2(\Omega;\M^{N\times N}),
\end{equation}
\begin{equation}
\label{min2}
\mathcal F(\v_0)=\min_{\v\in H^1(\Omega;\R^N)} \mathcal F(\v)=\lim_{{h_{j}}\to 0}\mathcal F_{h_{j}}(\v_j),
\end{equation}
\begin{equation}\label{convsqrth2}
\sqrt{h_{j}}\,\nabla\v_j\ \to \ \mathbf 0\quad\hbox{in} \quad L^2(\Omega;\M^{N\times N})
\end{equation}
and by \eqref{min2}, \eqref{convsqrth2}
\begin{equation}\label{convminbis}\begin{array}{ll}
&\displaystyle\min_{\v\in H^1(\Omega;\R^N)} \mathcal E(\v)\ge \mathcal F(\v_0)=\min_{\v\in H^1(\Omega;\R^N)} \mathcal F(\v)=\lim_{{h_{j}}\to 0}\mathcal F_{h_{j}}(\v_j)=\\
&\\
&\displaystyle\lim_{{h_{j}}\to 0}\int_{\om} \mathcal V_{0}\big( \mathbb E(\v_j) +\textstyle\frac{1}{2} h_j\nabla  \v_j^{T}  \nabla  \v_j\big)\,d\xx-\mathcal L(\v_j)=\\
&\\
&\displaystyle \int_{\om} \mathcal V_{0}\big(\mathbb E(\v_0)\big)\,d\xx-\mathcal L(\v_{0})=\mathcal E(\v_{0})
\end{array}
\end{equation}
that is $\v_{0}\in \argmin \mathcal E$.\\
Thanks to \eqref{bar1}, \eqref{proj} and \eqref{nablaP} we get
$$\displaystyle \int_{\om}(\x\times \v_{h_{j}})\,d\x= \int_{\om}(\x\times (\v_{h_{j}}-\frac{1}{|\om|}\int_{\om}\v_{h_{j}}\,d\x)\,d\x$$
which, thanks to \eqref{convsqrth2}, implies
\begin{equation}
\sqrt{h_{j}}\,\nabla(\mathbb P\v_j)\to \mathbf 0
\end{equation}
so that
\begin{equation}\lab{Bhj}
\mathbf B_{h_{j}}:= \frac{h}{2}\left\{\nabla(\mathbb P\v_j)^{T}\nabla(\mathbb P\v_j)+
 \nabla \v_j^{T}\nabla(\mathbb P\v_j)-\nabla(\mathbb P\v_j)^{T}\nabla \v_j\right\}\to \mathbf 0
\end{equation}
strongly in $L^2(\Omega;\M^{N\times N})$.
Since $\v_j$ is a minimizing sequence,  \eqref{coerc} and Poincar\'{e}-Korn inequality yield
\begin{equation}
\lab{stimsoprasotto}
\begin{array}{ll}
&\displaystyle\int_\Omega\vert \mathbb E(\v_j)+\textstyle\frac{1}{2}h_{j}\nabla\v_j^T\nabla\v_j\vert^2\,d\xx\le C+
\mathcal L(\v_j)=\\
&\\
 &\displaystyle C+\mathcal L(\v_j-\mathbb P\v_j)\le  C+C'\left(\int_\Omega\vert \mathbb E(\v_j)\vert^2\,d\xx\right)^{\frac{1}{2}},\\
\end{array}
\end{equation}
hence $\mathbf D_{h_{j}}:= \mathbb E(\v_j)+\textstyle\frac{1}{2}h_{j}\nabla\v_j^T\nabla\v_j$ are equibounded in $L^2(\Omega;\M^{N\times N})$
and by setting $\w_j:=\v_j-\mathbb P\v_j$, by recalling that $\mathbf B\to V_{0}(\mathbf B)$ is convex we have
\begin{equation}
\mathcal F_{h_{j}}(\v_j)-\mathcal F_{h_{j}}(\w_j)\ge \int_{\om}\mathbf B_{h_{j}}\cdot V_{0}'(\mathbf D_{h_{j}}+\mathbf B_{h_{j}})\,d\xx.
\end{equation}
Since $|V_{0}'(\mathbf B)|\le C|\mathbf B|$ for some $C> 0$, by \eqref{Bhj} and  \eqref{stimsoprasotto} we get
\begin{equation}
\left |\displaystyle \int_{\om}\mathbf B_{h_{j}}\cdot V_{0}'(\mathbf D_{h_{j}}+\mathbf B_{h_{j}})\,d\xx\right | \le C \int_{\om}\left (|\mathbf B_{h_{j}}|^{2}+\left |\mathbf B_{h_{j}}\right | \left | \mathbf D_{h_{j}}\right |\right )\,d\xx\to 0
\end{equation}
that is
\begin{equation}
\mathcal F_{h_{j}}(\v_j)\ge \mathcal F_{h_{j}}(\w_j)+o(1)
\end{equation}
which proves that $\w_j$ is a minimizing sequence too.  It is now readily seen that  $\w_j$ are equibounded in $H^{1}(\om;\mathbf R^{N})$ and \eqref{energconv} follows from  \eqref{convminbis} so the claim is  proven.
\end{proof}
\begin{remark}\label{rmk5.2}
\rm{By inspection of the proof, Theorem \ref{strongconvmin} holds true also
for more general energies:
e.g if $\mathcal W$ is a convex function of $\mathbf F^{T}\mathbf F-\mathbf I$ with quadratic growth,
which is finite if and only if $\det \mathbf F>0$.
}
\end{remark}

\end{document}